\documentclass[12pt]{article}
\usepackage{amsmath}
\usepackage{amssymb}
\usepackage{bbm}
\usepackage[margin=2.25cm]{geometry}
\usepackage{hyperref}
\usepackage{tikz}
\usepackage{setspace}

\usetikzlibrary { decorations.pathmorphing, decorations.pathreplacing, decorations.shapes, }

\newtheorem{thm}{Theorem}[section]
\newtheorem{cor}[thm]{Corollary}

\newtheorem{lemma}[thm]{Lemma}

\newtheorem{theorem}[thm]{Theorem}
\newtheorem{remark}[thm]{Remark}

\newtheorem{corollary}[thm]{Corollary}

\newenvironment{proof}{{\bf Proof:}}{\hfill$\square$\vskip.5cm}

\newcommand{\R}{\mathbb{R}}
\newcommand{\N}{\mathbb{N}}

\newcommand{\Z}{\mathbb{Z}

}

\renewcommand{\Z}{{\mathbb{Z}}}

\begin{document} 

\title{Asymptotics of the d'Arcais Numbers at Small $k$}
\author{Shannon Starr}
\date{23 January 2026}

\maketitle

\abstract{
The d'Arcais numbers are the triangular array $\{A(2,n,k)\, :\, n=0,1,\dots,\, k=0,\dots,n\}$, such that
$\sum_{n=0}^{\infty} \sum_{k=0}^{n} A(2,n,k) x^k z^n/n! = 
((z;z)_{\infty})^{-x}$.
The infinite $q$-Pochhammer symbol is $(q;q)_{\infty} = \prod_{n=1}^{\infty} (1-q^n)$.
Holding $k$ fixed and considering large $n$, we note that the ratio $k! A(2,n,k)/n!$ is asymptotic to 
$C(k) \sigma_{2k-1}(n)/n^k$
where the divisor sum function is
$\sigma_p(n) = \sum_{d|n} d^p$
and $C(k) = (\zeta(2))^k/(\Gamma(k) \zeta(2k))$.
This is a slightly generalized version
of one of Ramanujan's formulas from his paper, 
``On Certain Arithmetical Functions," 
and it is an  immediate consequence of the more recent article of
Oliver, Shreshta and Thorne.
Heim and Neuhauser made a conjecture, that $A(2,n,k)/A(2,n,k-1)$ is greater than or equal to $A(2,n,k+1)/A(2,n,k)$,
for $k=2,3,\dots$ and all $n$.
The conjecture is false for $k=2$, and it is true for $k=3,4,\dots$ when $n$ is sufficiently large.
We  consider the Hardy-Ramanujan circle method as a heuristic step.
}

\section{Introduction}

In a previous article we considered a large deviation formula for the d'Arcais numbers \cite{LDPdArcais}. 
The d'Arcais numbers may be thought of in the following way.
Consider the abundance index
$$
\sigma_{-1}(n)\, =\,  \sum_{d|n} \frac{1}{d}\, =\, \frac{\sigma_1(n)}{n}\, =\, \sum_{d|n} \frac{d}{n}\, .
$$
Note that the generating function is 
$$
\sum_{n=1}^{\infty} \frac{\sigma_1(n)}{n} z^n\, 
=\, \sum_{k=1}^{\infty} \sum_{r=1}^{\infty} \frac{z^{kr}}{k}\, =\, -\sum_{r=1}^{\infty} \ln(1-z^r)\, ,
$$
and this may be rewritten as $-\ln((z;z)_{\infty})$ where the infinite $q$-Pochhammer symbol is 
$$
(q;q)_{\infty}\, =\, \prod_{n=1}^{\infty}(1-q^n)\, .
$$
The d'Arcais numbers are the coefficients in the two variable generating function
$$
\sum_{n=0}^{\infty} \sum_{k=0}^{n} \frac{A(2,n,k)}{n!}\, z^n x^k\, =\, \left((z;z)_{\infty}\right)^{-x}\, .
$$
From this description, they may also be written as the
so-called Bell transform of ($n!$ times) the abundance index:
$$
A(2,n,k)\, =\, \frac{n!}{k!} \sum_{(\nu_1,\dots,\nu_k) \in \N^k} \mathbf{1}_{\{n\}}(\nu_1+\dots+\nu_k) \prod_{j=1}^{k} \left(\frac{\sigma_1(\nu_j)}{\nu_j}\right)\, .
$$
The reason for the parameter $2$ is because the d'Arcais numbers are part of a more general multi-dimensional combinatorial array which we describe now.

Let $\pi^{(1)},\dots,\pi^{(\ell)} \in S_n$ be $\ell$ permutations of $[n]$. Suppose that they all commute with one another.
Next, consider the subgroup $\langle \pi^{(1)},\dots,\pi^{(\ell)} \rangle$ generated by these permutations.
Consider the action of this subgroup on $[n]$.
Then let $\mathcal{K}(\pi^{(1)},\dots,\pi^{(\ell)})$ denote the number of disjoint orbits partitioning $[n]$, under this action.
Let $A(\ell,n,k)$ be the number of $\ell$-tuples $(\pi^{(1)},\dots,\pi^{(\ell)}) \in S_n^{\ell}$ such that $\pi^{(1)},\dots,\pi^{(\ell)}$ commute
and $\mathcal{K}(\pi^{(1)},\dots,\pi^{(\ell)})=k$.
Bryan and Fulman gave a generating function for these
$$
\sum_{n=0}^{\infty} \sum_{k=0}^{n} \frac{A(\ell,n,k)}{n!}\, x^k z^n\, =\, \prod_{d_1,\dots,d_{\ell-1}=1}^{\infty} \left(1-z^{d_1\cdots d_{\ell-1}}\right)
^{-x d_1^{\ell-2} d_2^{\ell-3} \cdot d_{\ell-2}^1}\, ,
$$
as a formal power series. It converges if $|z|<1$.
Note $A(\ell,n,0)=0$ if $n>0$ but $A(\ell,0,0)=1$.
Abdesselam, Brunialti, Doan and Velie gave a concrete representation by showing that $A(\ell,n,k)$ counts the number of discrete tori, with twists, of dimension $k$
with $n$ points \cite{ABDV}.
This is also related to the wreath product which Bryan and Fulman used.

There is a conjecture by Heim and Neuhauser \cite{HeimNeuhauser} that $k \mapsto A(2,n,k)$ is log-concave: $A(2,n,k)/A(2,n,k-1) \geq A(2,n,k+1)/A(2,n,k)$.
Abdesselam stated the new conjecture that $k \mapsto A(\ell,n,k)$ is log-concave, for $\ell=3,4,\dots$.
He has now proved both conjectures in a number of important regimes \cite{Abdesselam,AbdesselamSole}.
Some of Abdesselam's methods involve asymptotics.

Indeed, with Abdesselam, the author has an article about the weak central limit theorem for $A(\ell,n,k)$ in the ``typicality regime,''
$k \sim x \sqrt{n}$ for some $x \in (0,\infty)$ \cite{AbdesselamStarr}.
That result is too weak to allow a concrete result such as log-concavity. (It gives circumstantial evidence since the
logarithm of the Gaussian density function is concav)).
Abdesselam then used the theory of iterated contour integrals to obtain a much harder complete asymptotic expansion down to orders $o(1)$.
For example, he uses tools as in Pemantle, Wilson and Melczer \cite{PemantleWilsonMelczer}.
Here we only consider $\ell=2$ and we explain why now.

Note that for $\ell=2$, the generating function is 
$$
\sum_{n=0}^{\infty} \sum_{k=0}^{n} \frac{A(2,n,k)}{n!}\, x^k z^n\, =\, \prod_{d=1}^{\infty} (1-z^d)^{-x}\, =\,
\exp\left(-x\sum_{d=1}^{\infty} \ln(1-z^d)\right)
$$
Thus, in particular
$$
\sum_{n=1}^{\infty}
\frac{A(2,n,1)}{n!} z^n\, =\, -\sum_{d=1}^{\infty} \ln(1-z^d)\, =\, \ln\left(\frac{1}{(z;z)_{\infty}}\right)\, .
$$
This is closely related to the Dedekind eta function
$$
\eta(\tau)\, =\, e^{\pi i \tau/12} \prod_{n=1}^{\infty} \left(1-e^{2n\pi i \tau}\right)\, =\, q^{1/24} (q;q)_{\infty}\, ,
$$
for $q=e^{2\pi i \tau}$, on the domain $\tau \in \{x+iy\, :\, x\in \R\, ,\ y>0\}$, the upper-half-plane.
This puts us in the domain of analytic number theory.
For instance, $\eta$ satisfies the modular symmetries:
$$
\eta(\tau+1)\, =\, e^{\pi i/12} \eta(\tau)\ \text{ and }\ \eta\left(-\frac{1}{\tau}\right)\, =\, \sqrt{-i\tau}\, \eta(\tau)\, ,
$$
where the branch of the square-root is such that $\sqrt{-i\tau}=1$ when $\tau=i$. See for instance \cite{Apostol}.
Moreover, the latter is a type of Poisson summation formula. It relates values of $\eta$ near the cusp at $0$ to
values near the cusp at $+i\infty$, where $\eta$ is less singular.
So, using the modular symmetry, optimal asymptotics are available with minimal effort.

For past results showing the breadth of such asymptotics, we mention Moak \cite{Moak} as well as Banerjee and Wilkerson \cite{BanerjeeWilkerson}.
Another important tool is the so-called Bell transform (so-named by contributors to OEIS).
An excellent reference is Chapter 1 of Pitman's monograph \cite{Pitman}.
Using these tools, we were able to re-prove part of Abdesselam's results only using a single contour integral instead of two.
Therefore, our touchstone was Flajolet and Sedgewick \cite{FlajoletSedgewick}.

We did use the idea of the circle method of Hardy and Ramanujan, but not the full Rademacher decomposition.
Instead we settled for an asymptotic result, while Rademacher's decompsotion is an exact formula involving a convergent series.
This allowed for a much easier proof, not unlike Newman's simplification \cite{Newman}.
In particular, we only needed one major arc, near the real axis.
But our results were constrained to the scenario where $k \to \infty$ as $n \to \infty$, so that $k$ is a large parameter.
In that article, we noted that to handle small values of $k$, even just asymptotically,  would require consideration of all Farey fractions.

In this note, we consider that analysis.
It quickly leads to formulas.
But we prove those formulas in other ways.
In the next section, we describe those formulas.

\section{Main results}

The  divisor function is 
$$
\sigma_p(n)\, =\, \sum_{d|n} d^p\, .
$$
We will primarily consider negative integer $p$.
But there is a relation $\sigma_{-p}(n) = \sigma_p(n)/n^p$.
Let us denote
$$
\sigma_{p}^{(q)}(n)\, =\, n^q \sigma_{p}(n)\, .
$$
Then let
$$
S^{(a,b)}_{r,s}(n)\, =\, \sum_{k=1}^{n-1} \sigma_{r}^{(a)}(k) \sigma_{s}^{(b)}(n-k)\, 
=\, \sum_{k=1}^{n-1} k^{a} \sigma_{r}(k) (n-k)^b \sigma_{s}(n-k)\, ,
$$
which is the arithmetical convolution $\sigma^{(a)}_r * \sigma^{(b)}_s(n)$.
\begin{theorem}
\label{thm:main}
For $a,b\geq 0$ and $r,s\geq 1$, we have
$$
S^{(a,b)}_{-r,-s}(n)\, \sim\, \frac{\Gamma(a+1)\Gamma(b+1)}{\Gamma(a+b+2)} \cdot 
\frac{\zeta(r+1)\zeta(s+1)}{\zeta(r+s+2)}\, 
\sigma_{-r-s-1}^{(a+b+1)}(n)\, ,\ \text{ as $n \to \infty$.}
$$
\end{theorem}
This result is weaker than what is usually considered.
Ramanujan in \cite{Ramanujan}
and then Ingham \cite{Ingham}, Halberstam \cite{Halberstam1} 
and most recently Oliver, Shrestha and Thorne \cite{OliverShresthaThorne}
all considered the leading order remainder term as well.
We do not explore the remainder term, here.

If one took $a=r$ and $b=s$ then the result just says
$$
S_{r,s}(n)\, \sim\, \frac{\Gamma(r+1)\Gamma(s+1)}{\Gamma(r+s+2)} \cdot 
\frac{\zeta(r+1)\zeta(s+1)}{\zeta(r+s+2)}\, 
\sigma_{r+s+1}(n)\, ,\ \text{ as $n \to \infty$,}
$$
where
$$
S_{r,s}(n)\, =\, \sigma_r *\sigma_s(n)\, =\, \sum_{k=1}^{n-1} \sigma_r(k) \sigma_s(n-k)\, .
$$
In this form, this result is simply the original conclusion of Ramanujan.

The result is proved exactly as all of those results are. In particular Oliver, Shresthra and Thorne
divide the task into two steps, following Halberstam. The first step is to make the estimate
$$
\sum_{\substack{k \in \{1,\dots,n\} \\ k \equiv k_0 (\operatorname{mod} m)}} k^{\alpha-1} (n-k)^{\beta-1}\,
=\, \frac{n^{\alpha+\beta-1}}{m}\, B(\alpha,\beta) + O(n^{\alpha+\beta-2})\, ,
$$
for $\alpha,\beta\geq 1$, using the Beta integral $B(\alpha,\beta)$.
The leading order part is a power of $n$ depending on $\alpha$ and $\beta$,  times $m^{-1}$ times a constant.
The term $m^{-1}$ appears in the second part of their argument, but the factor $n^{\alpha+\beta-1}$ is taken outside the remaining sums.
Therefore, changing $\alpha$ and $\beta$ from $r+1$ and $s+1$ to instead equal $a+1$ and $b+1$ does not change the validity of the proof method.

With the goal of being self-contained, we will give details  of the proof in Section 4.
\begin{corollary}
\label{cor:alphaAsymp}
Define the quantities
$$
\alpha(n,k)\, =\, \frac{k!\, A(2,n,k)}{n!}\, =\,
 \sum_{(\nu_1,\dots,\nu_k) \in \N^k} \mathbf{1}_{\{n\}}(\nu_1+\dots+\nu_k) \prod_{j=1}^{k} \left(\frac{\sigma_1(\nu_j)}{\nu_j}\right)\, .
$$
For each $k \in 2,3,\dots$, we have
\begin{equation}
\label{eq:alphaAsymptotics}
\alpha(n,k)\, \sim\, \frac{\big(\zeta(2)\big)^k \sigma_{2k-1}(n)}{\Gamma(k) \zeta(2k) n^k}\, =\, 
 \frac{n^{k-1}\big(\zeta(2)\big)^k \sigma_{-2k+1}(n)}{\Gamma(k) \zeta(2k)}\, ,\ \text{ as $n \to \infty$.}
\end{equation}
\end{corollary}
\begin{proof}
The result may be rewritten as a formula for the $k$-fold arithmetical convolution
$$
\big(\sigma_{-1}^{(0)} * \cdots * \sigma_{-1}^{(0)}\big)(n)\, =\, 
\frac{\big(\zeta(2)\big)^k}{\Gamma(k) \zeta(2k)}\, \sigma_{-2k+1}^{(k-1)}\, .
$$
With this, it follows from Theorem \ref{thm:main} and induction on $k$.
\end{proof}

\begin{remark}
In our previous article, we gave a formula for the large deviation formula assuming $k \to \infty$.
If one looks carefully at that formula assuming $n/k \to \infty$, then one has
$$
\alpha(n,k)\, \sim\, \frac{e^{k \ln (\zeta(2)) - k \ln(k) + k + (k-1)\ln(n)}}{\sqrt{2\pi/k}}\, .
$$
This is precisely the $k \to \infty$ asymptotic expression of the above formula, if we replace $\sigma_{-2k+1}(n)$ by its Cesaro limit equivalent $\zeta(2k)$.
\end{remark}

In the interest of pedagogy, we include a detailed description of the induction step in Appendix \ref{sec:Induction}.
For now, let us assume that the reader sees how to fill in the steps for themselves.
Then
we have the following.
\begin{cor}
\label{cor:alpha}
For $k = 2$, 
$$
\liminf_{n \to \infty} \frac{\big(\alpha(n,k)\big)^2}{\alpha(n,k-1) \alpha(n,k+1)}\, =\, 0\, .
$$
For $k=3,4,\dots$, we have
$$
\liminf_{n \to \infty} \frac{\big(\alpha(n,k)\big)^2}{\alpha(n,k-1) \alpha(n,k+1)}\, =\,
\frac{k}{k-1} \cdot \frac{\zeta(2k+2)\zeta(2k-2)}{\big(\zeta(2k)\big)^2} \cdot \frac{\big(\zeta(2k-1)\big)^2}{\zeta(2k-3)\zeta(2k+1)}\, >\, 1.
$$
\end{cor}
\begin{proof}
From Corollary \ref{cor:alphaAsymp}, we know
$$
\frac{\big(\alpha(n,k)\big)^2}{\alpha(n,k-1) \alpha(n,k+1)}\, \sim\, \frac{k}{k-1}\, \cdot\,
\frac{\zeta(2k+2) \zeta(2k-2)}{\big(\zeta(2k)\big)^2}\, \cdot\, \frac{\big(\sigma_{-2k+1}(n)\big)^2}{\sigma_{-2k+3}(n) \sigma_{-2k-1}(n)}\, .
$$
But it is an easy fact  that
$$
\frac{\big(\sigma_{-2k+1}(n)\big)^2}{\sigma_{-2k+3}(n) \sigma_{-2k-1}(n)}\, >\, \frac{\big(\zeta(2k-1)\big)^2}{\zeta(2k-3) \zeta(2k+1)}\, ,
$$
and that
$$
\liminf_{n \to \infty} \frac{\big(\sigma_{-2k+1}(n)\big)^2}{\sigma_{-2k+3}(n) \sigma_{-2k-1}(n)}\, =\, \frac{\big(\zeta(2k-1)\big)^2}{\zeta(2k-3) \zeta(2k+1)}\, .
$$
For completeness, we explain this in Appendix \ref{sec:divisor}.
Of course, $\zeta(1)=\infty$: that proves the first statement when $k=2$ because $2k-3=1$.
But note that for $k=3,4,\dots$,
$$
\frac{k}{k-1} \cdot \frac{\zeta(2k+2)\zeta(2k-2)}{\big(\zeta(2k)\big)^2} \cdot \frac{\big(\zeta(2k-1)\big)^2}{\zeta(2k-3)\zeta(2k+1)}\, 
>\, 1\, .
$$
This can be seen in various ways.
For example, see Appendix \ref{sec:calculus}.
\end{proof}

\begin{remark}
The Okounkov-Nekrasov polynomials are closely related to the d'Arcais polynomials.
In \cite{HongZhang}, Hong and Zhang announced their proof of Heim and Neuhauser's conjecture for log-concavity of the coefficients
of the Okounkov-Nekrasov polynomials at small $k$.
Our method shares some common points with their proof.
They have explicit bounds on sufficient sizes for $n$, for their bounds to apply.
That is valuable and we do not provide that.
But we do provide an explicit formula for the limit inferior.
\end{remark}

\begin{remark}
The most interesting regime is $k=xn^{1/2}$ for $x \in (0,1)$. This is the typicality regime, as Abdesselam points out in \cite{Abdesselam}.
In that regime, he obtained a full asymptotic expansion such that $A(\ell,n,k) \sim \widetilde{A}(\ell,n,k)$ as $n \to \infty$, for an explicit formula
found by himself. He can also handle $\ell=3,4,\dots$ as well as $\ell=2$.
In a separate paper, the author considered just $\ell=2$ (as in this note) and considered large deviations assuming $1\ll k $ and $k/n < 1-\epsilon$
for a fixed $\epsilon>0$ \cite{LDPdArcais}.
The interesting regime of $k/n \to 1$ is also treated by Raghavendra Tripathi \cite{Tripathi} where he proves uniform log-concavity in an interesting way.
\end{remark}

In the next section we describe the leading order contribution from the Hardy and Ramanujan circle method \cite{HardyRamanujan}
when applied to the problem of describing the asymptotics of the d'Arcais numbers.
To prove that the circle method leads to the formulas above would require the full apparatus of Rademacher \cite{Rademacher}.
We do not carry that out, instead using the method of the articles by Ramanujan, by Ingham, by Halberstam and by Oliver, Shrestha and Thorne.
But it is useful to see how the formulae are deduced.

\section{Heuristics from the circle method}

Recall that
$$
\alpha(n,k)\, =\, \frac{k! A(2,n,k)}{n!}\, .
$$
Start with the Cauchy integral formula
$$
\alpha(n,k)\, =\, e^{k \ln \big(F(t)\big)+nt} \int_{-\pi}^{\pi} e^{-in\theta} \left(\frac{F(t-i\theta)}{F(t)}\right)^k\, \frac{d\theta}{2\pi}\, ,
$$
for $t>0$,
where
$$
F(t)\, =\, -\ln\big((e^{-t};e^{-t})_{\infty}\big)\, =\, \sum_{n=1}^{\infty} \frac{\sigma(n)}{n}\, e^{-nt}\, .
$$
Due to the modular symmetry of $\eta(\tau) = e^{\pi i \tau/2} \exp(-F(-2\pi i \tau))$, we know
$$
F(t)\, =\, \frac{\zeta(2)}{t} - \frac{\ln(t)}{2} + \frac{\ln(2\pi)}{2} - \frac{t}{24} + F\left(\frac{4\pi^2}{t}\right)\, ,
$$
which leads to useful asymptotics when $t \to 0^+$.
\begin{figure}
\begin{center}
\begin{tikzpicture}
\draw (0,0) node[] {\includegraphics[width=10cm,height=5cm]{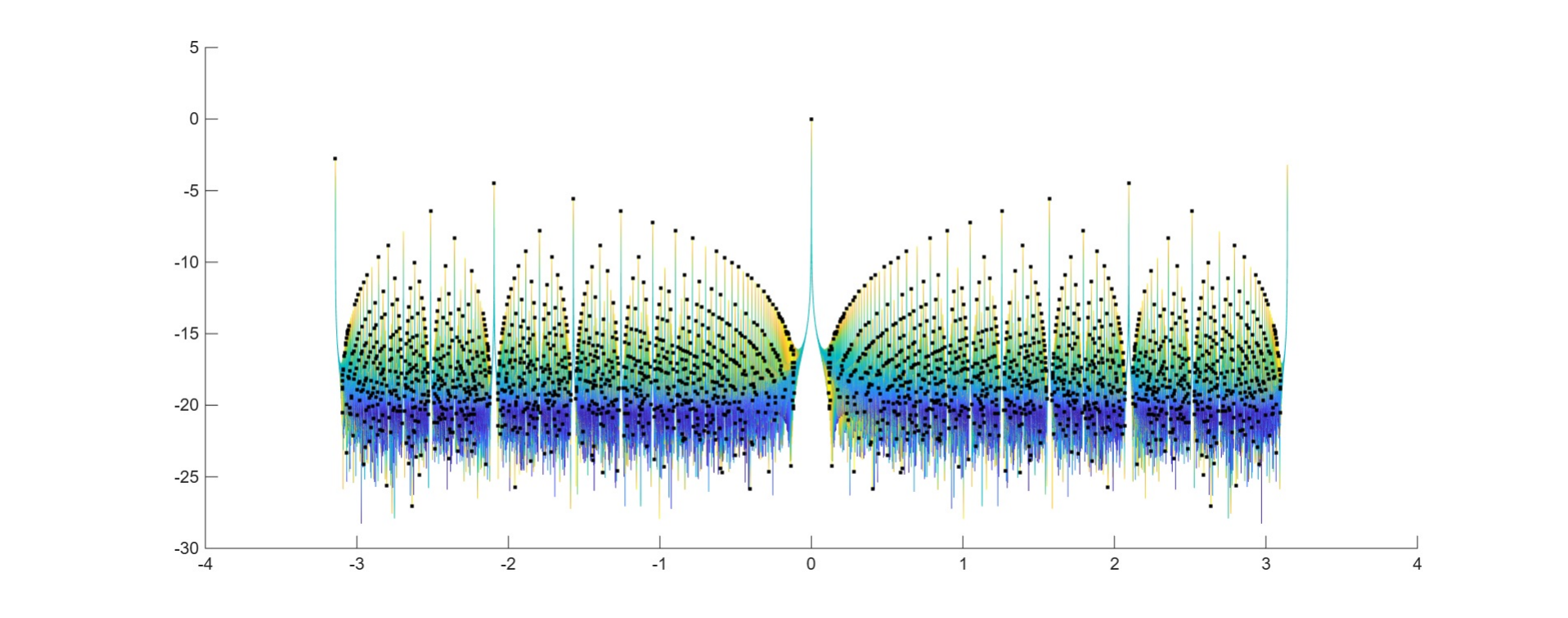}};
\draw (8,0) node[] {\includegraphics[width=8cm,height=5cm]{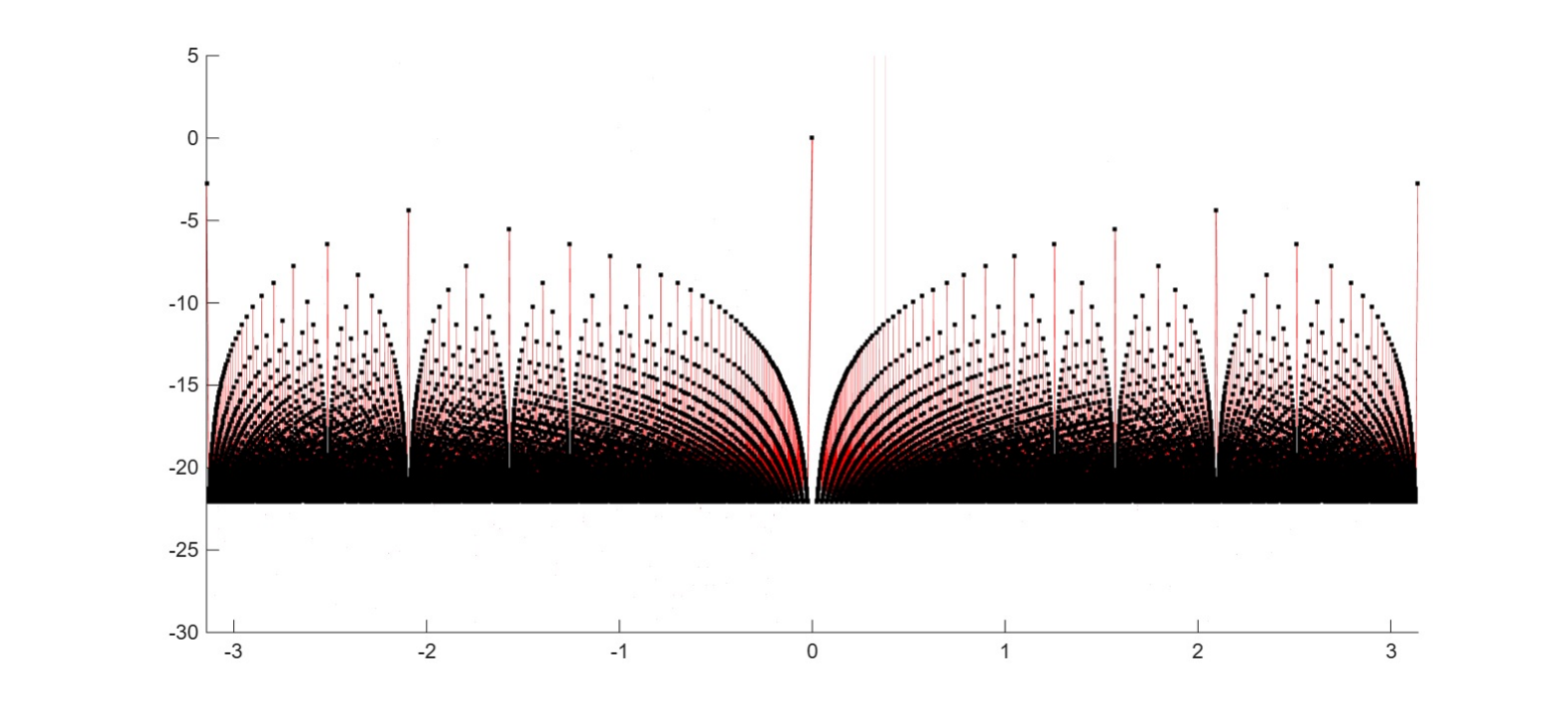}};
\end{tikzpicture}
\caption{
Here we plot two similar figures. On the left we have $\ln\left(|F(t-i\theta)|^2/\big(F(t)\big)^2\right)$ for $t=1/50,000$, calculated by including $250,000$ terms in the $n$-sum
for the power series of $F$. The colors represent  $\cos(\phi)$ when we write $F(t-i\theta)/F(t) = \rho e^{i\phi}$. The dots are placed whenever $\cos(\phi)>0.99$. On the right hand side, we have plotted $-4\ln(q)$ versus $p/q$ for the the Farey sequence with denominators up to 250, for comparison.
\label{fig:LogNormSquared}
}
\end{center}
\end{figure}
Due to this expansion, when choosing $t_{n,k}$ so that $-F(t_{n,k})/F'(t_{n,k}) = k/n$, we have $t_{n,k} \sim k/n$ as $n \to \infty$ (with $k$ fixed).
Then
$$
k \ln \big(F(t_{n,k})\big)+n t_{n,k}\, \sim\, k \ln(\zeta(2)) - k \ln(k/n)+k\, ,\ \text{ as $n \to \infty$.}
$$
That is the asymptotic formula for the constant prefactor to the contour integral.
But turning to the contour integral, we should focus on ``major arcs,'' and ``minor arcs'' of the integral.
The major arcs are centered at every Farey fraction.

Consider the ratio $F(t-i\theta)/F(t)$.
In Figure \ref{fig:LogNormSquared}, we plot $\ln(|F(t-i\theta)|^2/(F(t))^2)$ and the color represents $\cos(\phi)$ for $\phi$ such that $F(t-i\theta)/F(t) = \rho e^{i\phi}$.
(The  height of the curve is $\ln(\rho^2)$.)
Thus yellow represents points where the argument is nearly $\phi=0$, and blue represents points where the argument is nearly $\phi=\pi$.
We have also placed black dots at all the points where  $\cos(\phi)>0.99$, since these should represent the local maxima of $|F(t-i\theta)|/F(t)$
(as a function of $\theta$).
On the right plot in the figure, we have plotted points $(p/q,-4\ln(q))$ for all the Farey fractions in the sequence up to denominators equal to $250$.

The modular symmetry is such that, for a given integer matrix $(a,b;c,d)$ with $ad-bc=1$, we have
$$
\eta\left(\frac{a\tau+b}{c\tau+d}\right)\, =\, \epsilon(a,b;c,d) (c\tau+d)^{1/2} \eta(\tau)\, ,
$$
where
$$
\epsilon(a,b;c,d)\, =\, \begin{cases} e^{i\pi b/12} & \text{ if $c=0$, $d=1$,}\\
e^{i\pi(\frac{a+d}{12c}-s(d,c)-\frac{1}{4})} & \text{ if $c>0$,}
\end{cases}
$$
where $s(h,k)$ is the Dedekind sum
$$
s(h,k)\, =\, \sum_{n=1}^{k-1} \frac{n}{k} \left(\frac{hn}{k} - \left\lfloor \frac{hn}{k} \right\rfloor - \frac{1}{2}\right)\, .
$$
See, for example, Wikipedia.

Now suppose that $p$ and $q$ are relatively prime so that $p/q$ is a Farey fraction.
In order to move the cusp of $\eta$ at $p/q$ to the cusp at $+i\infty$, we take $c=q$ and $d=-p$.
Then, in order to obtain a unimodular matrix, we require
$$
-ap-bq=1\, .
$$
So, letting $x$ and $y$ be the integers from Euclid's algorithm $px+qy=1$, we have $a=-x$ and $b=-y$.
But even without calculating $(x,y)$ explicitly, it is easy to see that the leading order part of the asymptotics is 
$$
F\left(t(1-i\Theta)+\frac{2\pi i p}{q}\right)\, \sim\, \frac{\zeta(2)}{q^2 t(1-i\Theta)}\, ,\ \text{ as $t \to 0^+$,}
$$
where we have chosen to rescale $\theta = \frac{2\pi p}{q} + t \Theta$ in order to better capture the asymptotics.
Therefore,
$$
\frac{F\left(t(1-i\Theta)+\frac{2\pi i p}{q}\right)}{F(t)}\, 
\sim\, \frac{1}{q^2(1-i\Theta)}\, ,\ \text{ as $t \to 0^+$.}
$$
That explains the similarity in the left and right plots in Figure \ref{fig:LogNormSquared}.

Since this is raised to the power of $k$ in the integrand, we have
$$
\left(\frac{F\left(t(1-i\Theta)+\frac{2\pi i p}{q}\right)}{F(t)}\right)^k\, 
\sim\, \frac{1}{q^{2k}(1-i\Theta)^k}\, ,\ \text{ as $t \to 0^+$.}
$$
But it is also multiplied by $e^{-in\theta} = e^{-2\pi i n p/q} e^{-int\Theta}$.
Since $n t_{n,k} \sim k$, we consider the integral
$$
\int_{-\infty}^{\infty} \frac{e^{-ik\Theta}}{(1-i\Theta)^k}\, d\Theta\, =\, \frac{2\pi k^{k-1}}{e^k\, \Gamma(k)}\, ,
$$
which is formal when $k=1$ but rigorous when $k=2,3,\dots$.
Then recalling that we also have the prefactor of the contour integral
$$
e^{k \ln \big(F(t)\big)+nt}\, \sim\, \big(\zeta(2)\big)^k e^{-k \ln(k/n)} e^k\, ,
$$
we are led to the ansatz (because $d\theta/2\pi$ equals $ t d\Theta/2\pi \sim k d\Theta/(2\pi n)$) 
$$
\alpha(n,k)\, \sim\, \frac{n^{k-1}\big(\zeta(2)\big)^{k}}{\Gamma(k)} {\sum_{p/q}}' \frac{e^{-2\pi i np/q}}{q^{2k}}\, ,\ \text{ as $n \to \infty$.}
$$
where the summation is over Farey fractions. 

But this uses Ramanujan's sums.
As usual, define
$$
c_q(n)\, =\, \sum_{\substack{p \in \{1,\dots,q\} \\ 
(p,q)=1}} e^{2\pi i n p/q}\, .
$$
Then this says (taking the complex conjugate of our previous formula, because both sums are actually real)
$$
\alpha(n,k)\, \sim\, \frac{n^{k-1}\big(\zeta(2)\big)^{k}}{\Gamma(k)} \sum_{q=1}^{\infty} \frac{c_q(n)}{q^{2k}}\, ,\ \text{ as $n \to \infty$.}
$$
But $\sum_{q=1}^{\infty} c_q(n)/q^{2k}$ is known to equal $\sigma_{-2k+1}(n)/\zeta(2k)$.
So this leads to the ansatz
$$
\alpha(n,k)\, \sim\, \frac{n^{k-1} \big(\zeta(2)\big)^k}{\Gamma(k)} \cdot \frac{\sigma_{-2k+1}(n)}{\zeta(2k)}\, ,\ \text{ as $n \to \infty$.}
$$
Note that for $k=1$ this reads as an identity, since $\alpha(n,1) = \sigma_{-1}(n)$.

\section{Proof of the main result}

The idea of the proof is immediate from Oliver, Shrestha and Thorne. But we repeat their argument here, for completeness.

They start with the following lemma, which is Lemma 3.2 in their paper, which they attribute to Halberstam:
\begin{lemma}
Suppose $\alpha,\beta\geq 1$. Then
$$
\sum_{\substack{k \in \{1,\dots,n\} \\ k \equiv k_0 (\operatorname{mod} m)}} k^{\alpha-1} (n-k)^{\beta-1}\,
=\, \frac{n^{\alpha+\beta-1}}{m}\, B(\alpha,\beta) + O(n^{\alpha+\beta-2})\, .
$$
\end{lemma}
Recall that the Beta integral is $B(\alpha,\beta) = \int_0^1 x^{\alpha-1} (1-x)^{\beta-1}\, dx = \Gamma(\alpha) \Gamma(\beta)/\Gamma(\alpha+\beta)$.
The next result is their Lemma 3.3. But it is also a well-known result.
\begin{lemma}
For $r,s >1$, 
$$
\sum_{m=1}^{\infty} \sum_{\substack{n \in \N \\ (m,n)=1}} m^{-r} n^{-s}\, =\, \frac{\zeta(r)\zeta(s)}{\zeta(r+s)}\, .
$$
\end{lemma}
Then the rest of the argument follows from Oliver, Shrestha and Thorne's proof of their Theorem 3.1.

Write
$$
S_{-r,-s}^{(a,b)}(n)\, =\, \sum_{k=1}^{n-1} \sigma_{-r}(k) \sigma_{-s}(n-k) k^{a} (n-k)^{b}\, .
$$
Then expand out the divisor  functions
$$
S_{-r,-s}^{(a,b)}(n)\, =\, \sum_{k=1}^{n-1}  k^{a} (n-k)^{b} \sum_{d|k} d^{-r} \sum_{e|(n-k)} e^{-s}\, .
$$
Now interchange the order of summation
$$
S_{-r,-s}^{(a,b)}(n)\, =\,  \sum_{d=1}^{n-1} d^{-r} \sum_{e=1}^{n-1} e^{-s}
\sum_{\substack{k \in \{1,\dots,n-1\} \\ d|k \\ e|(n-k)}}   k^{a} (n-k)^{b} \, .
$$
Suppose that $(d,e) | n$, otherwise it is impossible to have any $k$ such that $d|k$ and $e|(n-k)$.
Then by  the Euclidean algorithm there is an $x,y \in \Z$ such that $xd+ye=(d,e)$.
Then choose $k_0 = xnd/(d,e)$. Then $k_0 + nye/(d,e) = n$. So $n-k_0 = nye/(d,e)$.
Thus we see $d|k_0$ and $e|(n-k_0)$ (because $n/(d,e)$ is an integer).
Now for any other $k$, writing $k=k_0 + (k-k_0)$ we see that the condition $d|k$ and $e|(n-k)$
means $d|(k-k_0)$ and $e|(k-k_0)$. Therefore, $k \equiv k_0 (\operatorname{mod} \frac{de}{(d,e)})$.
So
$$
S_{-r,-s}^{(a,b)}(n)\, =\,  \sum_{d=1}^{n-1} d^{-r} \sum_{e=1}^{n-1} e^{-s}
\sum_{\substack{k \in \{1,\dots,n-1\} \\ k \equiv k_0 (\operatorname{mod} \frac{de}{(d,e)})}}   k^{a} (n-k)^{b} \, .
$$
Therefore, by the first lemma
$$
S_{-r,-s}^{(a,b)}(n)\, =\,  \sum_{d=1}^{n-1} d^{-r} \sum_{e=1}^{n-1} e^{-s}
\left(n^{a+b+1}\, \left(\frac{(d,e)}{de}\right) B(a+1,b+1) +O(n^{a+b})\right)\, .
$$
Assuming $r,s\geq 1$, we have $\sum_{d=1}^{n-1} d^{-r} \leq \ln(n)$ and $\sum_{e=1}^{n-1} e^{-r} \leq \ln(n)$.
So
$$
S_{-r,-s}^{(a,b)}(n)\, =\,  O\left(n^{a+b} \big(\ln(n)\big)^2\right) + B(a+1,b+1) n^{a+b+1} \sum_{d=1}^{n-1} d^{-r} \sum_{e=1}^{n-1} e^{-s}
\mathbf{1}_{\N}\left(\frac{n}{(d,e)}\right)
\left(\frac{(d,e)}{de}\right)\, .
$$
At this point the argument is exactly as in Oliver, Shrestha and Thorne, because the power of $n$ is now outside the summation.
But we complete the argument in the interest of being self-contained.

Let us rename $w=(d,e)$. Then we can rewrite the sum as follows, letting $d=wi$ and $e=wj$,
$$
S_{-r,-s}^{(a,b)}(n)\, =\,  O\left(n^{a+b} \big(\ln(n)\big)^2\right) + B(a+1,b+1) n^{a+b+1} \sum_{w|n} \sum_{\substack{i,j \in \{1,\dots,(n/w)-1\}\\
(i,j)=1}} \frac{w^{-r-s} i^{-r} j^{-s}}{wij}\, .
$$
But then it is easy to see that
$$
\sum_{w|n} w^{-r-s-1} \sum_{\substack{i,j \in \{1,\dots,(n/w)-1\}\\
(i,j)=1}} i^{-r-1} j^{-s-1}\, =\, \sum_{w|n} w^{-r-s-1} \sum_{\substack{i,j \in \N \\ (i,j)=1}} i^{-r-1} j^{-s-1} + O(n^{-r}+n^{-s})\, ,
$$
where recall that we assumed $r,s>0$.
But then we see that the second lemma implies
$$
 \sum_{w|n} w^{-r-s-1} \sum_{\substack{i,j \in \N \\ (i,j)=1}} i^{-r-1} j^{-s-1}\, =\, \frac{\zeta(r+1) \zeta(s+1)}{\zeta(r+s+2)} \sum_{w|n} w^{-r-s-1}\, .
$$
This is $\zeta(r+s+2) \sigma_{-r-s-1}(n)$.

So then we have the desired result. More explicitly, we see that
\begin{equation*}
\begin{split}
S_{-r,-s}^{(a,b)}(n)\, 
&=\,  O\left(n^{a+b} \big(\ln(n)\big)^2\right) + O(n^{a+b+1-r} + n^{a+b+1-s}) \\
&\qquad
+ B(a+1,b+1)\, \frac{\zeta(r+1) \zeta(s+1)}{\zeta(r+s+2)}\, n^{a+b+1} \sigma_{-r-s-1}(n)\, .
\end{split}
\end{equation*}
Note that every divisor  function is bounded below by 1.
Therefore, this may be rewritten as 
$$
S_{-r,-s}^{(a,b)}(n)\, =\,  
B(a+1,b+1)\, \frac{\zeta(r+1) \zeta(s+1)}{\zeta(r+s+2)}\, n^{a+b+1} \sigma_{-r-s-1}(n) (1+ o(1))\, ,
$$
as $n \to \infty$.
This is the version of the asymptotic formula we claimed.

\section{Outlook and open problems}

A careful analysis of the remainder terms is important, and it was done for the original problem
considered by Ramanujan and then by Ingham, by Halberstam, and by Oliver, Shrestha and Thorne.
Our main interest is the asymptotic behavior of $\alpha(n,k)$ as in (\ref{eq:alphaAsymptotics}).
For this purpose we believe that an approach to the remainder term would be to complete the Rademacher analysis of the Hardy-Ramanujan circle method as begun in Section 3.
One advantage is that the modular symmetry is so accessible for the Dedekind-eta function.

If one wanted to apply the circle method to Theorem 2.1, then one would instead consider Eisenstein series
as well as their derivatives. These also have modular or near-modular properties that are amenable to the Rademacher series analysis.
We note that Halberstam did apply the circle method to Ramanujan's original formula in \cite{Halberstam2}.

We believe the most important open problem is to consider the $\ell=3,4,\dots$ sequence of the numbers $A(\ell,n,k)$,
to see if one can obtain asymptotics for small $k$.
For that, we would recommend as a starting point Abdesselam's important recent results \cite{AbdesselamSole}.
We note that his result could be considered to be a generalization of a result of Bringmann, Franke and Heim \cite{BFH}, 
in that he used a generalizations of Meinardus' theorem.
That is an approach to bypass using modularity or quasi-modularity, for example when that is not available.
That is the major impediment for our elementary approach.
We needed modularity to even obtain the conjecture using the first step of the Hardy-Ramanujan circle method.

\section*{Acknowledgments}
I am grateful to Raghav Tripathi for an interesting discussion with useful comments.
I am grateful to Malek Abdesselam for introducing me to this topic.

\appendix

\section{A detailed description of the induction step}
\label{sec:Induction}

Let us denote
$$
\beta(n,k)\, =\, 
\frac{n^{k-1}\big(\zeta(2)\big)^k \sigma_{-2k+1}(n)}{\Gamma(k) \zeta(2k)}\, .
$$
Then we are trying to prove that
$$
\alpha(n,k)\, \sim\, \beta(n,k)\, ,\ \text{ as $n \to \infty$.}
$$
Specializing to $k=1$ we note that $\alpha(n,1)=\beta(n,1)$ since $\alpha(n,k)$ equals the $k$-fold 
arithmetic convolution
$\sigma_{-1} * \cdots * \sigma_{-1}(n)$.

For $k=1$, the asymptotic equivalence is already established as an identity.
Now induct on $k$.
For the induction hypothesis, suppose $K \in \N$ and 
suppose we have proved the asymptotic equivalence for all $k\leq K$.
Note that $\sigma_{-2k+1}(n)\geq 1$ for all $k,n\in \N$.
So $\beta(n,k)$ is never $0$. 
Define $\rho(n,k)$ such that
$$
\rho(n,k)\, =\, \frac{\alpha(n,k)}{\beta(n,k)}\, ,
$$
so that $\alpha(n,k) = \rho(n,k) \beta(n,k)$.
Then, by the induction hypothesis,  the asymptotic equivalence statement implies the following.
For each $\epsilon>0$, there is a number $N(K,\epsilon) \in \N$ such that we have $| \rho(n,k) -1 | \leq\epsilon$
whenever we have both $k=1,\dots,K$ and $n\geq N(K,\epsilon)$.
Now let us use the Theorem to imply the induction step.

Considering $\alpha(n,K+1)$, the arithmetic convolution sum may be rewritten
$$
\alpha(n,K+1)\, =\, \sum_{m=1}^{n-1} \alpha(m,K) \alpha(n-m,1)\, =\,
\sum_{m=1}^{n-1} \rho(m,K) \rho(n-m,1) \beta(m,K) \beta(n-m,1)\, .
$$
Let us digress to note that
$$
\sum_{m=1}^{n-1} \beta(m,K) \beta(n-m,1)\, =\, C(K)C(1) \sum_{m=1}^{n-1} m^{K-1} \sigma_{-2K+1}(m) \sigma_{-1}(n-m)\, ,
$$ 
where $C(k)=\big(\zeta(2)\big)^k/(\Gamma(k)\zeta(2k))$.
But this means
$$
\sum_{m=1}^{n-1} \beta(m,K) \beta(n-m,1)\, =\, C(K)C(1) S^{(K-1,0)}_{-2K+1,-1}(n)\, ,
$$ 
using the notation as in the theorem.
So by the theorem, we do know
$$
\sum_{m=1}^{n-1} \beta(m,K) \beta(n-m,1)\, \sim\, C(K)C(1)\, \frac{\Gamma(K) \Gamma(1)}{\Gamma(K+1)} \cdot \frac{\zeta(2K)\zeta(2)}{\zeta(2K+2)}\, n^{K} \sigma_{-2K-1}(n)\, ,
$$
as $n \to \infty$. Since the right hand side equals $C(K+1) n^K \sigma_{-2K-1}(n)$, which is $\beta(n,K+1)$, this is the result we are aiming for.
But now we want to establish this for the actual sum 
$\sum_{m=1}^{n-1} \rho(m,K) \rho(n-m,1) \beta(m,K) \beta(n-m-1)$.

Let us define $\widetilde{\rho}(n,k,\epsilon)$ as
$$
\widetilde{\rho}(n,k,\epsilon)\, =\, \max(\{\min(\{\rho(n,k),1+\epsilon\}),1-\epsilon\})\, .
$$ 
So we have $\widetilde{\rho}(n,k,\epsilon) \in [1-\epsilon,1+\epsilon]$.
But also, for $n\geq N(K,\epsilon)$ we have $\widetilde{\rho}(n,k,\epsilon)=\rho(n,k)$.
So we may rewrite
\begin{equation*}
\begin{split}
\alpha(n,K+1)\, 
&=\, \sum_{m=1}^{n-1} \widetilde{\rho}(m,K,\epsilon)\widetilde{\rho}(n-m,1,\epsilon)  \beta(m,K) \beta(n-m,1) \\
&\qquad
+ \sum_{\max(\{m,n-m\}) \leq N(K,\epsilon)} \mathcal{E}(m,n,K,\epsilon)\, ,
\end{split}
\end{equation*}
where $\mathcal{E}(m,n,K,\epsilon)$ equals
$$
\Big(\rho(m,K) \rho(n-m,1) - \widetilde{\rho}(m,K,\epsilon) \widetilde{\rho}(n-m,1,\epsilon)\Big)  \beta(m,K) \beta(n-m,1) \, .
$$
From the digression above, and the bounds $(1-\epsilon)^2\leq  \widetilde{\rho}(m,K,\epsilon)\widetilde{\rho}(n-m,1,\epsilon)\leq (1+\epsilon)^2$,
we know that the first sum satisfies
$$
\limsup_{n \to \infty} \frac{1}{\beta(n,K+1)}\, \sum_{m=1}^{n-1} \widetilde{\rho}(m,K,\epsilon)\widetilde{\rho}(n-m,1,\epsilon)  \beta(m,K) \beta(n-m,1)\, \leq\, (1+\epsilon)^2\, ,
$$
and
$$
\liminf_{n \to \infty} \frac{1}{\beta(n,K+1)}\, \sum_{m=1}^{n-1} \widetilde{\rho}(m,K,\epsilon)\widetilde{\rho}(n-m,1,\epsilon)  \beta(m,K) \beta(n-m,1)\, \geq\, (1-\epsilon)^2\, .
$$

But the second sum is clearly bounded by 
$$
 \sum_{\max(\{m,n-m\}) \leq N(K,\epsilon)} \mathcal{E}(m,n,K,\epsilon)\, \leq\, N(K,\epsilon) \max_m |\mathcal{E}(m,n,K,\epsilon)|\, .
$$
By Robin's bound \cite{Robin}, and the induction hypothesis $\max_m |\mathcal{E}(m,n,K,\epsilon)| = O(n^{K-1}\ln(\ln(n)))$ as $n \to \infty$, holding
$\epsilon$ (and $K$) fixed. But since $\beta(n,K+1) = C(K+1) n^K \sigma_{-2K-1}(n)$, and since $\sigma_{-2K-1}(n)\geq 1$, we see that
$$
\lim_{n \to \infty} \frac{1}{\beta(n,K+1)}\, \sum_{\max(\{m,n-m\}) \leq N(K,\epsilon)} \mathcal{E}(m,n,K,\epsilon)\,  =\, 0\, .
$$

So, combining the bounds for the first and second sums, we have established that
$$
(1-\epsilon)^2\, \leq\, \liminf_{n \to \infty} \frac{\alpha(n,K+1)}{\beta(n,K+1)}\, \leq\, 
\limsup_{n \to \infty} \frac{\alpha(n,K+1)}{\beta(n,K+1)}\, \leq\, (1+\epsilon)^2\, .
$$
Since $\epsilon>0$ is arbitrary, this establishes the induction step.

\section{A review of the divisor  functions}
\label{sec:divisor}

Note that, if $n = p_1^{a_1} \cdots p_{\ell}^{a_{\ell}}$, then
$$
\sigma_d(n)\, =\, \prod_{j=1}^{\ell} \frac{p_j^{d(a_j+1)}-1}{p_j^{d}-1}
$$
so
\begin{equation*}
\begin{split}
\frac{\big(\sigma_{-2k+1}(n)\big)^2}{\sigma_{-2k+3}(n) \sigma_{-2k-1}(n)}\, 
&=\, \prod_{j=1}^{\ell} \frac{\left(1-p_j^{-(2k+1)}\right)\left(1-p_j^{-(2k-3)}\right)}{\left(1-p_j^{-(2k-1)}\right)^2}\\
&\qquad \cdot \prod_{j=1}^{\ell} \frac{\left(1-p_j^{(-2k+1)(a_j+1)}\right)^2}
{\left(1-p_j^{(-2k-1)(a_j+1)}\right)\left(1-p_j^{(-2k+3)(a_j+1)}\right)}
\end{split}
\end{equation*}
Let us consider these two factors separately $A(p_1,\dots,p_{\ell})$ and $B(p_1,a_1,\dots,p_{\ell},a_{\ell})$.
Starting with the latter, we note that 
$$
B(p_1,a_1,\dots,p_{\ell},a_{\ell})\, =\, 
\prod_{j=1}^{\ell} \frac{\left(1-p_j^{(-2k+1)(a_j+1)}\right)^2}
{\left(1-p_j^{(-2k+1)(a_j+1)}\right)^2-\left(p_j^{2(a_j+1)}+p_j^{-2(a_j+1)}-2\right)p_j^{(-2k+1)(a_j+1)}}
$$
Thus this is greater than 1. But if we fix $p_1,\dots,p_{\ell}$ and take the limit $a_1,\dots,a_{\ell} \to \infty$, then $B(p_1,a_1,\dots,p_{\ell},a_{\ell})$ converges to 1
(assuming $k\geq 1$).

Now considering the first factor,
$$
A(p_1,\dots,p_{\ell})\, =\, \prod_{j=1}^{\ell} \frac{\left(1-p_j^{-(2k-1)}\right)^2-\left(p_j^2+p_j^{-2}-2\right)p_j^{-(2k-1)}}{\left(1-p_j^{-(2k-1)}\right)^2}\, .
$$
In other words, it is a product of $\ell$ terms, each of which is less than 1.
Therefore, it is greater than the infinite product, which we obtain if we take  $p_1,\dots,p_{\ell}$ to equal the first $\ell$ primes and then take $\ell \to \infty$.
The infinite product is $\big(\zeta(2k-1)\big)^2/\big(\zeta(2k+1)\zeta(2k-3)\big)$.

\begin{remark}
If we take a sequence of increasingly high powers of increasingly large primorials, then we can approach the limits where both inequalities
are saturated. That is how we obtain the limit inferior result.
The limit inferior claim of  Corollary \ref{cor:alpha} comes from taking a sequence of increasingly high powers of increasingly large primorials.
Then we can approach the limits where both inequalities
are saturated. That is how we obtain the limit inferior result.
Consider $n = (p_m\#)^r$ for a sequence
$m$ and $r$ approaching $\infty$. 
\end{remark}

\begin{remark}
The most surprising result is that $\big(\alpha(n,2)\big)^2/\big(\alpha(n,1) \alpha(n,3)\big)$ has limit inferior equal to $0$.
Note that this means $\big(A(2,n,2)\big)^2 / \big(A(2,n,1)A(2,n,3)\big)$ has limit inferior equal to $0$ as well, since the latter merely equals the former times $(3! 1!)/(2!)^2$.
Similarly, for $k=3,4,\dots$ since the limit inferior of $\big(\alpha(n,k)\big)^2/\big(\alpha(n,k-1)\alpha(n,k+1)\big)$ is greater than 1, so is the limit inferior of 
$\big(A(2,n,k)\big)^2/\big(A(2,n,k-1)A(2,n,k+1)\big)$ since $((k+1)!(k-1)!)/(k!)^2=(k+1)/k>1$.
\end{remark}

\section{A calculus formula}
\label{sec:calculus}

Firstly, note that
$$
\ln\left(\frac{\zeta(2k+2)\zeta(2k-2)}{\big(\zeta(2k)\big)^2} \cdot \frac{\big(\zeta(2k-1)\big)^2}{\zeta(2k-3)\zeta(2k+1)}\right)\,
=\, -\sum_{p} \sum_{r=1}^{\infty} \frac{p^{-2kr}}{r}\, \left(p^r-1\right) \left(p^r-p^{-r}\right)^2\, .
$$
Hence,
$$
\ln\left(\frac{\zeta(2k+2)\zeta(2k-2)}{\big(\zeta(2k)\big)^2} \cdot \frac{\big(\zeta(2k-1)\big)^2}{\zeta(2k-3)\zeta(2k+1)}\right)\,
\geq\, -\sum_{p} \sum_{r=1}^{\infty} \frac{p^{-(2k-3)r}}{r}\, 
=\, -\ln\left(\zeta(2k+3)\right)\, .
$$
But, for instance using convexity (since the average of $x^{-2k+3}$ over the interval $[n-\frac{1}{2},n+\frac{1}{2}]$ is greater than or equal to its value at the midpoint of the interval) for $k=3,4,\dots$ we have
$$
\zeta(2k-3)\, =\, \sum_{n=1}^{\infty} \frac{1}{n^{2k-3}}\, \leq\, 1 + \int_{3/2}^{\infty} \frac{1}{x^{2k-3}}\, dx\, =\, 1 + \frac{1}{2(k-2)(3/2)^{2k-4}}\, .
$$
Thus, bounding $(3/2)^{2k-4} \geq 1$ (for $k\geq 2$)
$$
\frac{\zeta(2k+2)\zeta(2k-2)}{\big(\zeta(2k)\big)^2} \cdot \frac{\big(\zeta(2k-1)\big)^2}{\zeta(2k-3)\zeta(2k+1)}\, >\, \frac{1}{1 + \frac{1}{2k-4}}\, 
=\, \frac{2k-4}{2k-3}\, .
$$
So
$$
\frac{k}{k-1}\, \cdot\, \frac{\zeta(2k+2)\zeta(2k-2)}{\big(\zeta(2k)\big)^2} \cdot \frac{\big(\zeta(2k-1)\big)^2}{\zeta(2k-3)\zeta(2k+1)}\,
>\, \frac{k}{k-1} \cdot \frac{2k-4}{2k-3}\, =\, \frac{2k^2-4k}{2k^2-5k+3}\, ,
$$
and the right-most term is $\geq 1$ as long as $k\geq 3$ (because $2k^2-5k+3\leq 2k^2-4k$).

\baselineskip=12pt
\bibliographystyle{plain}

\end{document}